\documentclass[11pt]{article}

\usepackage{amsmath,amsthm,verbatim,amssymb,amsfonts,amscd, graphicx,color}
\usepackage{epstopdf}
\usepackage{pgf,tikz}
\usepackage{mathrsfs}
\usepackage{graphics}
\usepackage{hyperref}
\usetikzlibrary{arrows}
\newtheorem{theorem}{Theorem}

\newtheorem{proposition}{Proposition}

\newtheorem{definition}{Definition}
\newtheorem{example}{Example}

\begin{document}
\title{An Invitation to Singular Symplectic Geometry}

\author{Roisin Braddell, Amadeu Delshams, Eva Miranda, C\'{e}dric Oms, Arnau Planas\footnote{Roisin Braddell and Amadeu Delshams are partially supported by the Spanish MINECO-FEDER grant MTM2015-65715. Eva Miranda is partially supported by  the Catalan Institution for Research and Advanced Studies ICREA via the Icrea Academia 2016 program. Eva Miranda, C\'{e}dric Oms and Arnau Planas are partially supported by the  Ministerio de Econom\'{\i}a y Competitividad project with reference MTM2015-69135-P and by the Generalitat de Catalunya project with reference 2014SGR634. C\'{e}dric Oms is supported by the AFR outgoing Program of the  Luxembourg National Research Fund (FNR) (project GLADYSS).}}

\maketitle

\begin{abstract} In this paper we analyze in detail a collection of motivating examples to consider $b^m$-symplectic forms and folded-type symplectic structures. In particular, we provide models in Celestial Mechanics for every $b^m$-symplectic structure. At the end of the paper, we introduce the odd-dimensional analogue to $b$-symplectic manifolds: $b$-contact manifolds.
\end{abstract}

\section{Introduction}	

Symplectic geometry has provided the classical models for problems in physics. However, sometimes the symplectic setting is insufficient for one's purposes: for instance in parametric-dependent systems, that of Poisson geometry is more appropriate. In this paper we present a particular class of Poisson manifolds satisfying some transversality conditions: $b^m$-Poisson manifolds (also called $b^m$-symplectic manifolds) and we exhibit examples coming from celestial mechanics.

 We also explore their natural \lq\lq duals" which we call $m$-folded symplectic structures\footnote{ In \cite{gmw} this duality is explored in detail using a desingularization (or deblogging) technique.} via examples. We call both $b^m$-Poisson manifolds and $m$-folded symplectic structures, singular symplectic structures in the sense that a symplectic structure either goes to infinity or drops rank on a subset.

The evolution of a Hamiltonian dynamical system is given by the flow of a vector field defined typically on the cotangent bundle of a manifold. This vector field is determined by a smooth function on the phase space and the canonical symplectic form on the cotangent bundle. The symplectic nature of the system has many important consequences such as preservation of phase space volume. These systems also come equipped with a well-developed perturbation theory.

Despite the useful structure that Hamiltonian systems exhibit, occasionally this structure is disregarded, particularly when studying the evolution close to singularities of the system, which often have difficult and interesting phenomena. Many examples of this occur in systems of celestial mechanics. The singularities of celestial mechanics fall into two categories: collision singularities, where two or more bodies occupy the same position in the configuration space; and non-collision singularities, which include the escape of a body to infinity in finite time. Here, \lq\lq singular" changes of coordinates are employed, or the points about the singularity are \lq\lq blown-up". Due to the non-canonical nature of these transformations, the symplectic form and the traditional form of Hamilton's equations are not preserved. Often the symplectic structure is simply discarded, along with all the useful associated tools. A similar situation occurs for the so called \lq\lq point transformations" of physics, which change position coordinates without reference to the usual change to the conjugate momenta, which would render the change canonical. These changes occur for a variety of reasons, many times simply out of convenience, e.g. when the corresponding canonical change to the momenta results in a complicated Hamiltonian. However, recently a useful middle ground is being investigated. The canonical symplectic form, under these changes of coordinates, is transformed to a form which is symplectic almost everywhere. Systematic investigation of such forms, which include folded symplectic and $b^m$-symplectic forms is a current active area of research.

In this paper we give an introduction to these singular symplectic structures. We then consider several Hamiltonian systems, with a preferred eye placed in classical problems of celestial mechanics, where classical coordinate transformations result in a singular symplectic structure on the corresponding \lq\lq Hamiltonian" system. We attempt to show how freedom in choosing non-canonical coordinates can be used to produce different insights in the dynamics of the system.

Particularly, we will compare the point Levi-Civita and the canonical Levi-Civita transformations in the Kepler problem, and the McGehee change and its canonical counterpart in the manifold at infinity in the restricted three body problem. We will also see in the Kepler problem with generalized potential how to produce all kinds of $b^m$-symplectic structures and folded structures, and the possible interplay between them. At the end of the paper, we introduce the odd-dimensional analogue to $b$-symplectic manifolds, which turn out to be the \emph{$b$-contact manifolds}.

\textbf{Organization of this paper:}
After the introduction, we introduce the main objects of this paper by giving a review on $b$-symplectic geometry in Section \ref{Sec:Preliminaries}. Section \ref{Sec:Pointtransformation} exhibits the examples of the Kepler and the two fixed centre problem, where a non-canonical change of variables induces a certain degeneracy on the symplectic form. In Section \ref{Sec:Escapesingularity}, we prove that the manifold at infinity in the planar restricted three-body problem can be seen as a $b^3$-symplectic manifold. In Section \ref{Sec:Doublecollision}, we show that the double collision of two particles in a generalized potential produces examples of $b^k$-symplectic and $k$-folded symplectic structures for any $k$. We end the paper with Section \ref{Sec:bContact}, where we introduce the odd-dimensional analogue of $b$-symplectic manifolds.

\section{Preliminaries}\label{Sec:Preliminaries}

There is a one to one correspondence between symplectic forms and non-degenerate Poisson structures on a manifold. This section will focus on an exposition of structures which are the \lq\lq next best" case, i.e. manifolds where these structures are non-degenerate away from a hypersurface of the manifold and behave well on the singular hypersurface. Explicitly:

\begin{definition}
Let $(M^{2n},\Pi)$ be an oriented Poisson manifold such that the map
\begin{equation}\label{bp}
p\in M\mapsto(\Pi(p))^n\in\Lambda^{2n}(TM)
\end{equation}
is transverse to the zero section, then $Z=\{p\in M|(\Pi(p))^n=0\}$ is a hypersurface and
we say that $\Pi$ is a \textbf{$b$-Poisson structure} on $(M^{2n},Z)$ and $(M^{2n},Z,\Pi)$ is a \textbf{$b$-Poisson manifold}.
 The hypersurface $Z$ is called \textbf{singular hypersurface}.
\end{definition}

\begin{definition}
Let $(M^{2n},\omega)$ be a manifold with $\omega$ a  closed $2$-form such that the map
$$p\in M\mapsto(\omega(p))^n\in\Lambda^{2n}(T^*M)$$
is transverse to the zero section , then $Z=\{p\in M|(\omega(p))^n=0\}$ is a hypersurface and we say that $\omega$ defines a \textbf{folded symplectic structure} on $(M,Z)$ if additionally its restriction to $Z$ has maximal rank.  We call the hypersurface $Z$  \textbf{folding hypersurface} and
 the pair  $(M,Z)$ is a \textbf{folded symplectic manifold}.
\end{definition}

$b$-Poisson structures were originally classified in dimension $2$ by Radko \cite{radko}. Recently, beginning with \cite{guimipi, guimipi2}, there have been interesting developments in the dynamical and topological aspects of $b$-Poisson structures in higher dimensions, see \cite{frejlichmartinezmiranda,km,cavalcanti,marcutosorno2,gualtieri,gualtierietal,gmps}.

A Poisson structure of $b$-type $\Pi$ on a manifold $M^{2n}$ defines a symplectic structure on a dense set in $M^{2n}$. The set of points where the Poisson structure is not symplectic is a hypersurface of $M^{2n}$. It is possible to study these structures in the symplectic setting, using the language of $b$-cotangent bundles and $b$-forms, a construction first given in \cite{melrose} and further used in \cite{nestandtsygan}. To import techniques from symplectic geometry, we first need some $b$-geometry.

\subsection{b-Geometry}

The category of $b$-manifolds was originally developed by Melrose \cite{melrose} in the context of manifolds with boundary. However many of the definitions can be used almost directly, with boundary being replaced by a chosen hypersurface of the manifold.

\begin{definition}
A \textbf{$b$-manifold} $(M,Z)$ is an oriented manifold $M$ together
 with an oriented hypersurface $Z$. A \textbf{$b$-map} is a map
\begin{equation}
f:(M_1,Z_1)\rightarrow(M_2,Z_2)
\end{equation}
so that $f$ is transverse to $Z_2$ and $f^{-1}(Z_2)=Z_1$.
\end{definition}

The notions of $b$-manifolds and $b$-maps give a well defined category. The have interesting structures which are analogues of the usual structures on smooth manifolds. Here we give some definitions, beginning with:

\begin{definition}
A $b$-vector field is a vector field on $M$, which is everywhere tangent to $Z$.
\end{definition}

The space of $b$-vector fields is a Lie sub-algebra of the Lie algebra of vector fields on $M$. The crux of $b$-geometry is that these are also the sections of a vector bundle on $M$, the $b$-tangent bundle. Let $U$ be an open neighbourhood about $p \in Z$ and $f$ a defining function of $Z$ in $U$. There is an (intrinsically defined) vector field, tangent to $Z$, given by $f\frac{\partial}{\partial f}$. We can choose a coordinate chart on $U$ of the form $(f, x_2,\cdots, x_{n})$ in which the $b$-vector fields restricted to $U$ form a free $C^{\infty}$-module with basis
\begin{center}
$\left( f\frac{\partial}{\partial f}, \frac{\partial}{\partial x_2}, \cdots , \frac{\partial}{\partial x_{n}} \right)$.
\end{center}

In this way the space of $b$-vector fields defines a locally free $C^{\infty}$-module, and so a vector bundle on $M$. Away from the critical hypersurface this vector bundle is isomorphic to the usual tangent bundle on $M$ by a theorem of Serre--Swan \cite{swan}. On the critical hypersurface there is a surjective bundle morphism $\phi:{^bTM}|_Z\rightarrow TM|_Z$. The kernel of this surjection is the trivial line bundle generated by $f\frac{\partial}{\partial f}$, which we call the \textbf{normal} $b$-\textbf{vector field}.

One can define the $b$-cotangent bundle of a $b$-manifold as the dual of this $b$-tangent bundle, with local basis
\begin{center}
$\left( \frac{df}{f}, d x_2, \cdots , d x_{n} \right)$
\end{center}
where the form $\frac{df}{f}$ is the well defined one form on the $b$-tangent bundle dual to the normal $b$-vector field.

We can adapt usual constructions on smooth forms to $b$-forms. In particular we have the concept of a $b$-form of degree $k$, as a section of the vector bundle $^b\Omega^k(M)=\Lambda^k(^bT^*M)$. Given a defining function $f$ for the critical hypersurface it is possible to write every $b$-form of degree $k$ as the sum
\begin{equation}
\omega=\alpha\wedge\frac{df}{f}+\beta,\ \alpha\in\Omega^{k-1}(M),\ \beta\in\Omega^{k}(M)
\end{equation}
so we can extend the exterior differential operator $d$ as an operator

\begin{equation}
d\omega=d\alpha \wedge \frac{df}{f}+d\beta.
\end{equation}

This allows us to define the notions of a \emph{closed} $b$-form and \emph{exact} $b$-form, and so, of $b$-de Rham complex and $b$-de Rham cohomology. It can be computed in terms of the cohomology of $M$ and $Z$, see \cite{guimipi2} for details. Explicitly

\begin{theorem}(Mazzeo-Melrose)
$^bH^*(M)\cong H^*(M)\oplus H^{*-1}(Z)$.
\end{theorem}

After a brief foray into $b$-geometry we are now ready to translate $b$-Poisson manifolds to the symplectic setting.

\subsection{b-Symplectic forms}

\begin{definition}\label{bsymp}
Let $(M^{2n},Z)$ be a $b$-manifold,  where $Z$ is the critical hypersurface as in Definition \ref{bp}. Let $\omega\in\,^b\Omega^2(M)$ be a closed $b$-form. We say that $\omega$ is \textbf{$b$-symplectic} if $\omega_p$ is of maximal rank as an element of $\Lambda^2(\,^b T_p^* M)$ for all $p\in M$.
\end{definition}

Using Moser's trick and adjusting some classical results from symplectic geometry, we get the corresponding Darboux theorem for the $b$-symplectic case:

\begin{theorem}[\textbf{$b$-Darboux theorem, \cite{guimipi2}}]\label{theorem:bDarboux}
Let $\omega$ be a $b$-symplectic form on $(M^{2n},Z)$. Let $p\in Z$. Then we can find a local coordinate chart $(x_1,y_1,\ldots,x_n,y_n)$ centered at $p$ such that hypersurface $Z$ is locally defined by $y_1=0$ and
${\omega=dx_1\wedge\frac{d y_1}{y_1}+\sum_{i=2}^n dx_i\wedge dy_i.}$
\end{theorem}

It can be shown that the a two-form on a $b$-manifold is $b$-symplectic if and only if its dual bi-vector field is $b$-Poisson. The dual of the $b$-Darboux theorem gives a local normal form of type
\begin{equation}
\Pi = y_1\frac{\partial}{\partial x_1}\wedge
\frac{\partial}{\partial y_1}+\sum_{i=2}^n \frac{\partial}{\partial x_i}\wedge
\frac{\partial}{\partial y_i}.
\end{equation}

Current research in these $b$-symplectic forms has resulted in \emph{topological} results, such as cohomological restrictions on the existence of $b$-symplectic structures and \emph{dynamical} results, including an extension of the notions of action-angle coordinates and a KAM theorem for $b$-symplectic forms \cite{kms}. These results become particularly interesting due to the discovery of these structures in equations coming from celestial mechanics, most notably arising from singularities of the solutions, where a traditional symplectic geometric description of the dynamics does not exist. First examples arising from celestial mechanics are explained in \cite{dkm}.

Another direction of the research has been to generalize these structures and consider more degenerate singularities of the Poisson structure. This is the case of $b^m$-Poisson structures \cite{scott} for which $\omega^m$ has a singularity of $A_m$-type in Arnold's list of simple singularities \cite{arnold, arnold2}. In the same spirit we may consider other singularities in this list.

As it happens with $b$-Poisson structures, it is possible and convenient  to consider a dual approach in their study and work with forms. We define:

\begin{definition} A {\bf symplectic $b^m$-manifold}  is a pair $(M^{2n}, Z)$ with a closed $ b^m $-two form $ \omega $  which has maximal rank  at every $p \in M$.
\end{definition}

Similar to the $b$-symplectic case, there exists a $b^m$-Darboux proved in \cite{gmw}

\begin{theorem}[\textbf{$b^m$-Darboux theorem, \cite{gmw}}]\label{theorem:bnDarboux}
Let $\omega$ be a $b^m$-symplectic form on $(M^{2n},Z)$ and $p\in Z$. Then we can find a coordinate chart $(x_1,y_1,\ldots,x_n,y_n)$ centered at
$p$ such that the hypersurface $Z$ is locally defined by $y_1=0$ and
$$\omega=d x_1\wedge\frac{d y_1}{y_1^m}+\sum_{i=2}^n d x_i\wedge d y_i.$$
\end{theorem}

In the same way, dually we obtain a $b^m$-Darboux form for $b^m$-Poisson bivector fields,
\begin{equation}
\Pi = y_1^m\frac{\partial}{\partial x_1}\wedge
\frac{\partial}{\partial y_1}+\sum_{i=2}^n \frac{\partial}{\partial x_i}\wedge
\frac{\partial}{\partial y_i}.
\end{equation}

We refer the reader to \cite{scott} and \cite{gmw} for details on the construction and properties of these structures.

\subsection{Folded symplectic forms}

A second class of important geometrical structures that model some problems in celestial mechanics are \emph{folded symplectic structures}. These are closed $2$-forms on even dimensional manifolds which are non-degenerate on a dense set thanks to the following transversality condition.

\begin{definition}
Let $(M^{2n},\omega)$ be a manifold with $\omega$ a closed $2$-form such that the map
$$p\in M\mapsto(\omega(p))^n\in\Lambda^{2n}(T^*M)$$
is transverse to the zero section , then $Z=\{p\in M|(\omega(p))^n=0\}$ is a hypersurface and we say that $\omega$ defines a \textbf{folded symplectic structure } on $(M,Z)$ if additionally its restriction to $Z$ has maximal rank.  We call the hypersurface $Z$  \textbf{folding hypersurface} and the pair  $(M,Z)$ is a \textbf{folded symplectic manifold}.
\end{definition}
The normal form of folded symplectic structures was studied by Martinet \cite{martinet}.

 \begin{theorem}[\textbf{folded-Darboux theorem, \cite{martinet}}]\label{theorem:foldedDarboux}
Let $\omega$ be a folded symplectic form on $(M^{2n},Z)$ and $p\in Z$. Then we can find a local coordinate chart $(x_1,y_1,\ldots,x_n,y_n)$ centered at
$p$ such that  the hypersurface $Z$ is locally defined by $y_1=0$ and
$$\omega=y_1d x_1\wedge{d y_1}+\sum_{i=2}^n d x_i\wedge d y_i.$$
\end{theorem}

 In analogy to the case of $b^m$-symplectic structures we define a new class of folded structures, namely \textbf{$m$-folded symplectic structures} for which $\omega^n$ has singularities of $A_m$-type in Arnold's list of simple singularities \cite{arnold2}, i.e. the top wedge power of $\omega$ has a local normal form of type  $\omega^n = y_1^m dx_1\wedge\dots\wedge dy_n$.

\subsection{Desingularization of $b^m$-symplectic forms} \label{Sec:Deblogging}

An immediate natural question to ask is whether we can associate a honest symplectic structure on a  $b^m$-symplectic manifolds. If the answer to this question is positive and we have a explicit control on this construction, then $b^m$-symplectic geometry is not far from actual symplectic manifolds. In \cite{gmw} this question is answered obtaining a surprising result: given a  $b^{2k}$-symplectic form we can associate a family of symplectic structures that converge to the initial $b^{2k}$-symplectic form, in the sense that these sympletic forms agree with the $b^{2k}$ outside an increasingly smaller neighbourhood of the critical set. This is called the ``desingularization" of the $b^{2k}$-sympletic form. In \cite{gmw} the odd counterpart is also considered, replacing the symplectic structure by a folded symplectic structure. This result connects $b^m$-symplectic geometry with symplectic and folded symplectic geometry.

 We now briefly recall how the desingularization is defined and the main result in \cite{gmw} for  $b^{2k}$-symplectic forms.

Any $b^{2k}$-form can be expressed as:
$$\omega = \frac{dx}{x^{2k}}\wedge \left(\sum_{i=0}^{2k-1}x^i\alpha_i\right) + \beta.$$

\begin{definition}The \textbf{$f_\epsilon$-desingularization} $\omega_\epsilon$ form associated to the $b^{2k}$-form $\omega$ is
	$$\omega_\epsilon = df_\epsilon \wedge \left(\sum_{i=0}^{2k-1}x^i\alpha_i\right) + \beta.
	$$
	where $f_\epsilon(x)$ is defined as $\epsilon^{-(2k-1)}f(x/\epsilon)$ and $f \in \mathcal{C}^\infty(\mathbb{R})$ is an odd smooth function satisfying $f'(x) > 0$ for all $x \in \left[-1,1\right] $ and
\begin{equation}
f(x) = \begin{cases}
\frac{-1}{(2k-1)x^{2k-1}}-2& \text{for} \quad x < -1,\\
\frac{-1}{(2k-1)x^{2k-1}}+2& \text{for} \quad x > 1.\\
\end{cases}
\end{equation}
\end{definition}

\begin{theorem}[\textbf{Desingularization}, \cite{gmw}]
	The $f_\epsilon$-desingularized form $\omega_{\epsilon}$ is symplectic. The family $\omega_\epsilon$ coincides with the $b^{2k}$-form $\omega$ outside an $\epsilon$-neighbourhood of $Z$. The family of bivector fields $\omega_\epsilon^{-1}$ converges to the structure $\omega^{-1}$ in the $C^{2k-1}$-topology as $\epsilon \rightarrow 0$.
\end{theorem}

An immediate consequence of this result is that a manifold admitting $b^{2k}$-symplectic structure also admits a symplectic form. In the case where $m$ is odd, we get a family of folded symplectic structures.

\section{Point transformations and Singular Symplectic Forms}\label{Sec:Pointtransformation}

Structures which are symplectic almost everywhere can arise as the result of a non-canonical changes of coordinates. Given configuration space $\mathbb{R}^2$ and phase space $T^*\mathbb{R}^2$ as is seen, for example, in the Kepler problem,  the traditional (canonical) Levi-Civita transformation is the following: identify $\mathbb{R}^2\cong\mathbb{C}$ so that $T^*\mathbb{R}^2\cong T^*\mathbb{C}\cong \mathbb{C}^2$ and treat $(q,p)$ as complex variables $(q_1+iq_2:=u,p_1+ip_2:=v)$ . Take the following change of coordinates $(q,p) = (u^2/2, v/\bar{u})$, where $\bar{u}$ denotes the complex conjugation of~$u$. The resulting coordinate change can easily be seen to be canonical. However this canonical change of coordinates can result in more difficult equations of motion, or a more difficult Hamiltonian, which can both obscure certain aspects of the dynamics of the system.

\subsection{The Kepler Problem}

In suitable coordinates in $T^*\left(\mathbb{R}^2\setminus\{0\}\right)$, the Kepler problem has Hamiltonian
\begin{equation}
 H(q,p)=\frac{\|p\|^2}{2}-\frac{1}{\|q\|}.
\end{equation}
With the canonical Levi-Civita transformation $(q,p) = (u^2/2, v/\bar{u})$, this becomes
\begin{equation}
 H(u,v)=\frac{\|v\|^2}{2\|\bar{u}\|^2}-\frac{1}{\|u\|^2}.
\end{equation}

Sometimes, as in this case, canonical changes lead to a more difficult system, so it may be desirable to leave the momentum unchanged and examine instead the transformation $(q,p) = (u^2/2, p)$ which can result in a simpler Hamiltonian. Now the transformation is not a symplectomorphism and the symplectic form on $T^*\mathbb{R}^2$ pulls back under the transformation to a two-form symplectic almost everywhere, but degenerate on a hypersurface of $T^*\mathbb{R}^2$.\\
Explicitly, the Liouville one-form $p_1 dq_1+p_2 dq_2=\Re(p d\bar{q})$ pulls back to
\begin{eqnarray*}
    \theta=\Re\left(p d\left(\frac{\bar{u}^2}{2}\right)\right)&=&\Re\left(p \bar{u}d\bar{u}\right)\\
    &=&p_1(u_1du_1-u_2du_2)+p_2(u_2du_1+u_1du_2)
\end{eqnarray*}
and computing $-d\theta$ we get the almost everywhere symplectic form
\begin{eqnarray*}
   \omega=u_1du_1\wedge dp_1 - u_2 du_1 \wedge dp_2 + u_2 du_2 \wedge dp_1 + u_1 du_2\wedge dp_2.
\end{eqnarray*}
Wedging this form with itself we find
\begin{eqnarray*}
   \omega\wedge\omega=(u_1^2-u_2^2) du_1 \wedge dp_1 \wedge du_2\wedge dp_2
\end{eqnarray*}
which is degenerate along the hypersurface given by $u_1=\pm u_2$.

\subsection{The Problem of Two Fixed Centers}

Related to the folded symplectic form found in the Levi-Civita transformation is the folded form associated with elliptic coordinates, employed while regularizing the problem of two fixed centers. This describes the motion of a satellite moving in a gravitational potential generated by two fixed massive bodies. We assume also that the motion of the satellite is restricted to the plane in $\mathbb{R}^3$ containing the two massive bodies. The Hamiltonian in suitable coordinates is given by
\begin{equation}
    H=\frac{p^2}{2m}-\frac{\mu}{r_1}-\frac{1-\mu}{r_2}
\end{equation}
where $\mu$ is the mass ratio of the two bodies (i.e. $\mu=\frac{m_1}{m_1+m_2})$.

Euler first showed the integrability of this problem using \emph{elliptic} coordinates, where the coordinate lines are confocal ellipses and hyperbola. Explicitly, consider a coordinate system in which the two centers are placed at $(\pm 1, 0)$, in which the (Cartesian) coordinates are given by $(q_1, q_2)$. Then the elliptic coordinates of the system are given by
\begin{align}
q_1&=\sinh\lambda\cos\nu\\
q_2&=\cosh\lambda\sin\nu
\end{align}
for $(\lambda, \nu)\in \mathbb{R}\times S^1$. Thus lines of $\lambda=c$ and $\nu=c$ are given by confocal hyperbola and ellipses in the plane, respectively. Similar to the Levi-Civita transformation this results in a double branched covering with branch points at the centers of attraction.

Pulling back the canonical symplectic structure $\omega=dq \wedge dp$ we find
\begin{equation}
\omega= \cosh\lambda \cos\nu (d\lambda \wedge dp_1+d\nu \wedge dp_2) -\sinh\lambda\sin\nu(d\nu \wedge dp_1 + d\lambda \wedge d p_2)
\end{equation}
which is degenerate along the hypersurface $(\lambda,\nu)$ satisfying $\cosh\lambda\cos\nu=\sinh\lambda\sin\lambda$.

\section{Escape Singularities and $b$-symplectic forms}\label{Sec:Escapesingularity}

The restricted elliptic 3-body problem describes the behavior of a massless object in the gravitational field of two massive bodies, orbiting in elliptic Keplerian motion. The planar version assumes that all motion occurs in a plane. The associated Hamiltonian of the particle is given by
\begin{equation}
H(q,p)=\frac{\|p\|^2}{2}+\frac{1-\mu}{\|q-q_1\|}+\frac{\mu}{\|q-q_2\|}=T+U
\end{equation}
where $\mu$ is the reduced mass of the system.

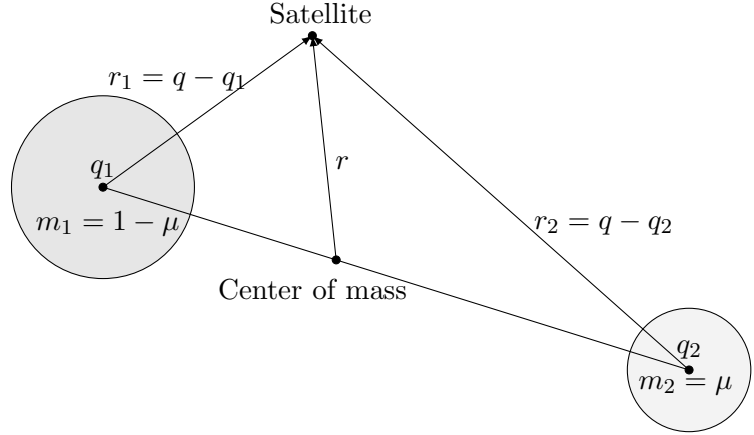
\begin{figure}
	
	\tikzset{>=latex}
	\centering
	\definecolor{xdxdff}{rgb}{0.49019607843137253,0.49019607843137253,1.}
	\definecolor{qqqqff}{rgb}{0.,0.,1.}
	\begin{tikzpicture}[line cap=round,line join=round,x=0.8cm,y=0.8cm]
	\clip(-6,-2.8) rectangle (7,4.5);
	\draw [fill=black,fill opacity=0.1] (-3.78,1.4) circle (1.2160508465447132cm);
	\draw [fill=black,fill opacity=0.05] (5.96,-1.64) circle (0.8209750300709516cm);
	\draw [->] (5.96,-1.64) -- (-0.3,3.92);
	\draw [->] (-3.78,1.4) -- (-0.3,3.92);
	\draw (5.96,-1.64)-- (-3.78,1.4);
	\draw [->] (0.09567095224776079,0.1903450005304732) -- (-0.3,3.92);
	\draw[color=black] (-3.7,0.8) node {$m_1 = 1 - \mu$};
	\draw[color=black] (5.9,-1.9) node {$m_2 = \mu$};
	\draw [fill=black] (-0.3,3.92) circle (1.5pt);
	\draw [fill=black] (-3.78,1.4) circle (1.5pt);
	\draw [fill=black] (5.96,-1.64) circle (1.5pt);
	\draw[color=black] (-3.78,1.7) node {$q_1$};
	\draw[color=black] (5.96,-1.3) node {$q_2$};
	\draw[color=black] (-0.16,4.29) node {$\text{Satellite}$};
	\draw[color=black] (4.54,0.8) node {$r_2 = q - q_2$};
	\draw[color=black] (-2.54,3.2) node {$r_1 = q - q_1 $};
	\draw [fill=black] (0.09567095224776079,0.1903450005304732) circle (1.5pt);
	\draw[color=black] (-0.3,-0.3) node {$\text{Center of mass}$};
	\draw[color=black] (0.2,1.8) node {$r$};
	\end{tikzpicture}
	\caption{Squeme of the three body problem.}
\end{figure}

After making a change to polar coordinates $(q_1,q_2)=(r\cos\alpha,r\sin\alpha)$ and the corresponding canonical change of momenta we find the Hamiltonian
\begin{equation}
H(r,\alpha,P_r,P_\alpha)=\frac{P_r^2}{2}+\frac{P_\alpha ^2}{2r^2}+U(r\cos\alpha,r\sin\alpha)
\end{equation}
where $P_r,P_\alpha$ are the associated canonical momenta and $U(r\cos\alpha,r\sin\alpha)$ is the potential energy of the system in the new coordinates.

The McGehee change of coordinates is traditionally employed to study the behavior of orbits near infinity, see also \cite{dkrs}. This non-canonical change of coordinates is given by
\begin{equation}\label{eqn:McGehee}
r=\frac{2}{x^2}.
\end{equation}
The corresponding change for the canonical momenta is easily seen to be
\begin{equation}
P_r=-\frac{x^3}{4}P_x.
\end{equation}
The Hamiltonian is then transformed to
\begin{equation}
H(r,\alpha,P_r,P_\alpha)=\frac{x^6P_x^2}{32}+\frac{x^4P_\alpha^2}{8}+U(x,\alpha).
\end{equation}
By dropping the condition that the change is canonical and simply transforming the position coordinate~(\ref{eqn:McGehee}), we are left with a simpler Hamiltonian, however the pull-back of the symplectic form under the non-canonical transformation is no longer symplectic, but rather $b^3$-symplectic:
\begin{equation}
\omega= \frac{4}{x^3} dx \wedge dP_r + d\alpha\wedge d P_\alpha.
\end{equation}

\section{$b^m$-Symplectic models for any $m$: McGehee coordinates in double collision}\label{Sec:Doublecollision}

 The system of two particles moving under the influence of the generalized potential $U(x) = -|x|^{-\alpha}$, $\alpha > 0$, where $|x|$ is the distance between the two particles, is studied by McGehee in \cite{mcgehee}. We fix the center of mass at the origin and hence can simplify the problem to the one of a single particle moving in a central force field.

In this section we prove
\begin{theorem} The McGehee change of coordinates used to study collisions provides  $b^m$-symplectic and  $m$-folded symplectic forms for any $m$ in the problem of a particle moving in a central force field with general potential depending on $m$.
\end{theorem}

The equation of motion writes down as
\begin{equation}
\ddot{x} = -\nabla U(x) = -\alpha |x|^{-\alpha-2}x
\end{equation}
where the dot represents the derivative with respect to time. In the Hamiltonian formalism, this equation becomes
\begin{equation}
\begin{array}{rcl}
\dot{x} & = &  y, \\
\dot{y} & = & -\alpha |x|^{-\alpha-2}x.
\end{array}
\end{equation}
To study the behavior of this system, the following change of coordinates is suggested in \cite{mcgehee}:
\begin{equation}\label{eq:mcgeheechange}
\begin{array}{rcl}
x & = & r^\gamma e^{i\theta}, \\
y & = & r^{-\beta\gamma}(v + iw)e^{i\theta}
\end{array}
\end{equation}
where the parameters $\beta$ and $\gamma$ are related with $\alpha$ in the following way:
\begin{equation}\label{eq:relations}
\begin{array}{rcl}
\beta & = & \alpha/2, \\
\gamma & = & 1/(1 + \beta).
\end{array}
\end{equation}
Identifying once more the plane $\mathbb{R}^2$ with the complex plane $\mathbb{C}$, we can write the symplectic form of this problem as $\omega = \Re (dx\wedge  d\overline{y})$.

\begin{proposition} Under the coordinate change (\ref{eq:mcgeheechange}), the symplectic form $\omega$ is sent to
\begin{enumerate}
\item a symplectic structure for $\alpha = 2/3$,
\item a $b$-symplectic structure for $\alpha = 2$,
\item a $b^2$-symplectic structure for $\alpha = 6$ and
\item a $b^3$-symplectic structure for $\alpha \rightarrow \infty$.
\end{enumerate}
\end{proposition}

\begin{proof}
The proof of this proposition is a straightforward computation. In the new coordinates, we obtain 
\begin{equation}\label{eq:omegachanged}
\begin{array}{rcl}
\omega = \Re(dx \wedge d\bar{y}) & = & \gamma r^{-\beta\gamma + \gamma - 1} dr\wedge dv - \gamma(1 - \beta)r^{-\beta \gamma + \gamma -1}wdr\wedge d\theta \\
&-& r^{-\beta\gamma + \gamma} dw \wedge d\theta.
\end{array}
\end{equation}
Wedging this form, we obtain
\begin{equation}
\begin{array}{rcl}
\omega\wedge \omega  & = &  -\gamma r^{-2\beta\gamma + 2\gamma - 1}dr\wedge dv \wedge d\theta\wedge dw\\
& = & \displaystyle -\gamma r^{\frac{2 -3\alpha}{2 + \alpha}}dr\wedge dv \wedge d\theta\wedge dw.
\end{array}
\end{equation}
where we use (\ref{eq:relations}). Let us set $f(\alpha) = \frac{2 -3\alpha}{2 + \alpha}$. We see that this function does not take values lower than $-3$ or higher than $1$. We easily see that we obtain 
\begin{enumerate}
\item a symplectic structure for $\alpha = 2/3$,
\item a $b$-symplectic structre for $\alpha = 2$,
\item a $b^2$-symplectic structure for $\alpha = 6$ and
\item a $b^3$-symplectic structure for $\alpha \rightarrow \infty$.
\end{enumerate}
\end{proof}

Note that $\omega$ tends not to a folded-symplectic form when $\alpha \rightarrow 0$, since the pullback of $\omega$ to the critical set vanishes.

The relations given by Eqs. (\ref{eq:relations}) are imposed in order simplify the equations, but dropping the second relation gives us enough freedom on the choice of parameters to obtain any $b^k$-symplectic structure or to a $k$-folded-symplectic for any $k$.

\begin{proposition} Under the change given by Eqs. (\ref{eq:mcgeheechange}) and the relations $\alpha = 2\beta$ and $\gamma = -\frac{k+1}{(\alpha + 2)}$, the symplectic form $\omega$ is sent to a $b^k$-symplectic form for $k$ positive, and for any value of $\alpha$. For $k$ negative the symplectic form $\omega$ is sent to a $(-k)$-folded-symplectic for $k$ negative only if $k = -1$ or $\alpha = 2$. 
\end{proposition}

\begin{proof}
Substituting the change of variables at the form $\omega$, we arrived at Eq. (\ref{eq:omegachanged}). The equations of the motion in the new coordinates are given by $\iota_{X_H}\omega = -dH$ where $H = \frac{1}{2}r^{-2\beta\gamma}(v^2 + w^2) - r^{-\alpha\gamma}$. By substituting and simplifying we obtain
\begin{equation}
\begin{array}{rcl}
\dot{\theta} & = & wr^{-\beta\gamma-\gamma}, \\
\dot{r} & = & \frac{1}{\gamma} v r^{-\beta\gamma-\gamma +1}, \\
\dot{w} & = & (\beta -1)wvr^{-\beta\gamma-\gamma}, \\
\dot{v} & = & -\beta v^2 r^{-\beta\gamma-\gamma} - w^2 r^{-\beta\gamma-\gamma} + \alpha r^{\gamma(\beta - \alpha) - \gamma}. \\
\end{array}
\end{equation}
We further simplify this equations by doing the following change in time:
\begin{equation}
d\tau = r^{\beta\gamma + \gamma}dt,
\end{equation}
which gives rise to the equations
\begin{equation}
\begin{array}{rcl}
\theta' & = & w, \\
r'& = & \frac{1}{\gamma} v r, \\
w' & = & (\beta -1)wv, \\
v' & = & -\beta v^2  - w^2 + \alpha r^{\gamma(2\beta - \alpha)}, \\
\end{array}
\end{equation}
where the $'$ denotes the derivative with respect the new time $\tau$. 
In order to integrate those equations, we want the last two equations to be independent of the two first, which only involve $v$ and $w$. Hence we impose that $\beta = \frac{\alpha}{2}$ and obtain
\begin{equation}
\begin{array}{rcl}
\theta' & = & w, \\
r'& = & \frac{1}{\gamma} v r, \\
w' & = & (\beta -1)wv, \\
v' & = & -\beta (v^2 -2)  - w^2. \\
\end{array}
\end{equation}
The last two equations can be solved using that $|w||v^2 + w^2 - 2|^{1-\beta}$ is an integral for this equations as in \cite{mcgehee}.

Recall that $\omega\wedge\omega = -\gamma r^{-\alpha\gamma + 2\gamma -1} dr \wedge dv \wedge dw \wedge d\theta$. Hence choosing $\gamma$ such that $-\alpha\gamma + 2\gamma -1 = k$ for a given $k \in \mathbb{Z^{+}}$ finishes the proof. This is done by taking $\gamma = \frac{k+1}{2-\alpha}$.

Observe that for $k \in \mathbb{Z^{-}}$ we would obtain a $k$-folded-symplectic form if and only if the pullback of the form to the critical set does not vanish. The pullback of $\omega$ to de critical set is $-r^{-\beta\gamma + \gamma}dw\wedge d\theta$. This is either $0$ or not well-defined unless $-\beta\gamma + \gamma = 0$. This is equivalent to asking $-\frac{k+1}{\alpha + 2}(1 - \beta) = 0$. This happens only if $\alpha = 2$ or $k = -1$. This concludes the proof.
\end{proof}

\section{$b$-Contact Geometry}\label{Sec:bContact}

We finish the article by giving some insight in one of our subsequent papers \cite{MO}. Contact geometry is often considered to be the \lq\lq odd-dimensional analogue of symplectic geometry".

\begin{definition}
	Let $(M,Z)$ be a (2n+1)-dimensional $b$-manifold. A $b$-contact structure is the Stefan--Sussmann distribution given by the kernel of a one $b$-form $\xi=\ker \alpha \subset {^b}TM$, $\alpha \in {^b\Omega^1(M)}$, that satisfies $\alpha \wedge (d\alpha)^n \neq 0$ as a section of $\Lambda^2(^bT^*M)$. We say that $\alpha$ is a $b$-contact form and the triplet $(M,Z,\xi)$ a $b$-contact manifold.
\end{definition}

Away from the critical set $Z$, the definition of $b$-contact coincides with the one of usual contact geometry. Hence, on $M\setminus Z$, the distribution is non-integrable, whereas on the critical set, due to the definition of the $b$-tangent bundle, the distribution is everywhere tangent to $Z$.

\begin{example} \label{extendedphasespace}
	Let $(M,Z)$ be a $b$-manifold of dimension $n$. Let $z,y_i, i=2,\dots,n$ be the local coordinates for the manifold $M$ on a neighbourhood at a point in $Z$, with $Z$ defined locally by $z=0$ and $x_i, i=1,\dots,n$ be the fiber coordinates on ${^{b}T^*M}$. Then the canonical one $b$-form is given in these coordinates by
	$$\lambda=x_1\frac{dz}{z}+\sum_{i=2}^{n}x_idy_i.$$
	The bundle $\mathbb{R}\times{^{b}T^*M}$ is a $b$-contact manifold with $b$-contact structure defined as the kernel of the one $b$-form
	$$dt+x_1\frac{dz}{z}+\sum_{i=2}^{n}x_idy_i,$$
	where $t$ is the coordinate on $\mathbb{R}$. The critical set is given by $\tilde{Z}=Z\times \mathbb{R}$. Direct computations yield that $\alpha\wedge (d\alpha)^n \neq 0$. Away from $\tilde{Z}$, $\xi = \ker \alpha $ is a non-integrable hyperplane field distribution, as in usual contact geometry. On the critical set however, $\xi$ is tangent to $\tilde{Z}$. This comes from the definition of $b$-vector fields. Since the rank of $\xi$ can drop by one on $\tilde{Z}$, we cannot say that $\xi$ is a hyperplane field.
\end{example}

To a $b$-contact form, one can associate a $b$-vector field $R_\alpha$, the Reeb vector field, defined by the equations
$$\begin{cases}
\iota_{R_\alpha}\alpha=1 \\
\iota_{R_\alpha}d\alpha=0.
\end{cases}$$
Those equations uniquely define a vector field: $d\alpha$ is a bilinear, skew-symmetric $2$-form on the space of $b$-vector fields ${^bTM}$, hence the rank is an even number. As $\alpha \wedge (d\alpha)^{n}$ is non-vanishing and of maximum degree, the rank of $d\alpha$ must be $2n$. Its kernel is $1$-dimensional and $\alpha$ is non-trivial on that line field. So a unique global vector field is defined by the normalization condition.

Symplectic and contact manifolds are closely related.  Indeed, it is well-known that to every contact manifold $(M,\alpha)$, one can associate a symplectic manifold $M\times \mathbb{R}$ by considering the symplectic form $d(e^t\alpha)$.

Going from symplectic to contact is also possible by the following construction. Let $(W,\omega)$ be a symplectic manifold. Recall that a Liouville vector field $X$ is defined by $\mathcal{L}_X \omega = \omega$, where $\mathcal{L}$ denotes the Lie derivative. We say that a hypersurface $H$ of $(W,\omega)$ is of \emph{contact type} if there exists a Liouville vector field $X$ transverse to $H$. Indeed, it is easy to check that the one-form $i_X\omega$ is a contact form on $H$.

This remains true in $b$-geometry, as we will prove in one of our subsequent papers \cite{MO}. This will generate further examples of $b$-contact manifolds.

\begin{example}
	The unit cotangent bundle of a $b$-manifold have a natural $b$-contact structure. Let $(M,Z)$ be a $b$-manifold of dimension $n$ with coordinates $z,x_i, i=2,\dots,n$ as in Example \ref{extendedphasespace}. The cotangent bundle has a natural $b$-symplectic structure defined by the $b$-form given by the exterior derivative of the Liouville one-form $\lambda$. The unit $b$-cotangent bundle is given by ${^bT^*_1M}=\{(q,p)\in T^*M |\  \|p\|=1 \}$, where $\| \cdot \|$ is the usual Euclidean norm. The vector field $\sum_{i=1}^n p_i \partial_{p_i}$ defined on the $b$-cotangent bundle ${^bT^*M}$ is a Liouville vector field, and is tranverse to the unit $b$-cotangent bundle, and hence induces a $b$-contact structure on it.
\end{example}

\section{Acknowledgments}

We are indebted to Andreas Knauf for enlightening conversations during the visit of Eva Miranda to Erlangen in December 2016 and in particular for directing our attention to the paper \emph{Double collisions for a classical particle system with nongravitational interactions} by McGehee \cite{mcgehee} to construct physical models for any $b^m$-symplectic and $m$-folded symplectic structure.

Eva Miranda thanks the Scientific Committee of the  XXV International Fall Workshop on Geometry and Physics for inviting her as speaker to the conference and for the invitation to  contribute to this volume. She also thanks the organizing committee for such a great conference. The results contained in this paper address several questions posed in her talk in particular that of providing physical examples for any $b^m$-symplectic structure.

\begin{small}

\end{small}

\end{document}